\documentclass[reqno]{amsart}
\usepackage{}
\usepackage[top=1in, bottom=1in, left=1in, right=1in]{geometry}
\usepackage{cases}
\usepackage[colorlinks,
            linkcolor=red,
            anchorcolor=blue,
            citecolor=green
            ]{hyperref}

\newtheorem{theorem}{Theorem}[section]

\theoremstyle{definition}
\newtheorem{definition}[theorem]{Definition}

\newtheorem{prop}[theorem]{Proposition}

\newtheorem{lem}[theorem]{Lemma}
\newtheorem{cor}[theorem]{Corollary}
\newtheorem{quest}[theorem]{Question}
\theoremstyle{remark}

\newcommand{\SL}{\mbox{\rm SL}}
\newcommand{\Sp}{\mbox{\rm Sp}}

\newcommand{\Tr}{\mbox{\rm Tr}}
\numberwithin{equation}{section}

\begin{document}\large

\title{Derivations of Siegel modular forms from connections}

\author{Enlin Yang}
\address{Department of Mathematical Science, Tsinghua University, Beijing, P. R. China 100084}
\email{yangenlin0727@126.com}

\author{Linsheng Yin}
\address{Department of Mathematical Science, Tsinghua University, Beijing, P. R. China 100084}
\email{lsyin@math.tsinghua.edu.cn}





\begin{abstract}
We introduce a method in differential geometry to study the
derivative operators of Siegel modular forms. By determining the
coefficients of the invariant Levi-Civita connection on a Siegel
upper half plane, and further by calculating the expressions of the
differential forms under this connection, we get a non-holomorphic
derivative operator of the Siegel modular forms. In order to get a
holomorphic derivative operator, we introduce a
weaker notion, called modular connection, on the Siegel upper half
plane than a connection in differential geometry. Then we show that
on a Siegel upper half plane there exists at most one holomorphic
modular connection in some sense, and get a possible holomorphic
derivative operator of Siegel modular forms.
\end{abstract}

\maketitle

\section*{Introduction}
In this paper, we introduce a differential geometric method to study the
derivative operators of Siegel modular forms, which, theoretically,
may be applied to the study of the derivative operators of any
automorphic form. Our idea comes from the observation on the two
derivative operators of the classical modular forms constructed by
combinations. It is well-known \cite{Zagier} that if $f$ is a
modular forms of weight $2k$, then
$D_kf:=\frac{df}{dz}-\frac{\sqrt{-1}k}{y}f$ is a non-holomorphic
modular forms of weight $2k+2$, and
$D_kf:=\frac{df}{dz}-\sqrt{-1}kG_2(z)f$, due to J.-P. Serre \cite{Serre}, is a
holomorphic modular forms of weight $2k+2$, where $G_2(z)$ is the
Eisenstein series of weight 2. We notice that the first operator can
be constructed by the Levi-Civita connection corresponding to the
invariant metric in the classical upper half plane, but the second
can not be constructed from any connection. However, if we loosen some
condition in the definition of the connection and define a concept
called modular connection, we can get Serre's holomorphic derivative
from the unique holomorphic modular connection on the upper half plane.
In this paper we
extend these results to Siegel upper half planes and Siegel modular
forms. We determine the coefficients of the Levi-Civita
connection corresponding to the invariant metric in a Siegel upper
half plane, and compute the expressions of the differential
forms under the connection, which give us a non-holomorphic
derivative operator of Siegel modular forms. Our main results are as follows.

Let $\mathbb H_g$ be the Siegel upper plane of degree $g$,
$\{dZ_{ij}:1\le i,j\le g\}$ a series of coordinates on $\mathbb H_g$,
$\Gamma_g=\Sp(2g,\mathbb Z)$
the full Siegel modular group which acts on $\mathbb H_g$ naturally, $M_k=M_k(\Gamma_g)$ the vector
space of the classical (or scalar-valued) Siegel modular forms of weight $k$,
$\widetilde{M}_k=\widetilde{M}_{k}(\Gamma_g)$ the $\mathbb C^\infty$-Siegel modular forms of weight $k$. Put
\[
\frac{\partial}{\partial Z}=(\partial_{ij})_{g\times g}\qquad\text{and}\qquad \partial_{ij}
=\frac{1}{2^{1-\delta(i,j)}}\cdot\frac{\partial}{\partial Z_{ij}},
\]
where $\delta(i,j)=1$ if $i=j$ and $\delta(i,j)=0$ if $i\neq j$.
\begin{theorem}[See theorem \ref{mainthm}]\label{thm1} Let $f\in M_{2k}(\Gamma_g)$. Then
\[
\det\left(\left[\frac{\partial}{\partial Z}-\sqrt{-1}kY^{-1}\right]f\right)\in \widetilde{M}_{2gk+2}(\Gamma_g).
\]
\end{theorem}
Here $\Gamma_g$ can be replaced by any congruence subgroup and the weight $2k$ can also be replaced by any
positive integer, with a little modification of our proofs.

To get a holomorphic derivative operator, we introduce the notion
of modular connections on the Siegel upper half plane, whose
condition is weaker than the classical definition of the
connections. Then we show the following result.
\begin{theorem}[See theorem \ref{mainthm1}]\label{thm2}
Any symmetric $g\times g$ matrix $G(Z)=(G_{ij}(Z))$ consisting of $\mathbb C^\infty$ functions on $\mathbb H_g$,
which satisfies the transformation formula
\[
(CZ+D)^{-1}G(\gamma(Z))=G(Z)\cdot(CZ+D)^t+2C^t
\]
for any $\gamma=\left(\begin{array}{cc} A & B\\ C& D\end{array}\right)\in\Sp(2g,\mathbb Z)$,
gives a unique modular connection $\mathbb D$ such that for any $\mathbb C^\infty$-function $f$ on $\mathbb H_g$
\[
\mathbb D(dZ_{rs})=-\sum_{i,j=1}^gG_{ij}dZ_{si}dZ_{rj}
\quad\text{and}\quad
\mathbb D(f(\det(dZ)^k)=\Tr\left(\left[\frac{\partial}{\partial
Z}-kG\right]f dZ\right)(\det(dZ))^k,
\]
and thus gives a derivative operator $M_{2k}\rightarrow\widetilde{M}_{2kg+2}$ by
$f\mapsto\det\left(\left[\frac{\partial}{\partial
Z}-kG\right]f\right)$.
Furthermore, there exists at most one holomorphic
symmetric matrix $G$ to satisfy the transformation formula. If such a $G$ exists,
the operator corresponding to $G$ is holomorphic.
\end{theorem}
In the classical case of $g=1$, the function $\sqrt{-1}G_2(z)$ is the unique holomorphic function on
the upper half plane satisfying the condition, which gives Serre's derivative. But when $g\ge 2$ we are not able
to construct such a matrix function $G$.

H. Maass has constructed a non-holomorphic derivative operator of Siegel modular forms by invariant
differential operators. For Siegel modular forms $f$ of weight $k$, Maass (\cite{Maass2}, P317) defines the operator
\[
D_k f(Z)=\det(Y)^{\kappa -k-1}\det\left(\frac{\partial}{\partial Z}\right)[\det(Y)^{k+1-\kappa} f(Z)],
\]
where $\kappa=(g+1)/2$ and the determinant of
$\frac{\partial}{\partial Z}$ is taken first, and shows that the
differential operator $D_k$ acts on the $\mathbb C^{\infty}$-Siegel
modular forms and maps $\widetilde{M}_k$ to $\widetilde{M}_{k+2}$.
We do not know the relation between our operator in Theorem \ref{thm1} and Maass'.
Compared to our operator, $D_k$ is linear with
respect to $f$. Moreover, our operator is a combination of degree 1
partial derivatives of $f$, but $D_k$ is a combination of degree $g$
partial derivatives. G. Shimura \cite{Shimura1} considers the
compositions $D_r^k=D_{r+2k-2}\cdots D_{r+2}D_r$ of Maass' operator,
which maps $\widetilde{M}_r$ to $\widetilde{M}_{r+2k}$. For our
operator one can also consider the compositions and then construct
the Rankin-Cohen brackets. We wish that Maass' operator could be got
in this way.

The paper is organized as follows. In section one, we introduce
the concept of modular connection on a Siegel upper plane, and show several lemmas
on it. In section two, we compute the expressions of the differential
forms under the modular connection, and prove the two theorems above.
Finally in section three, we show Lemma \ref{lemma:mainLemma} which
explicitly gives the connection
coefficients of the Levi-Civita connection on a Siegel upper half plane.

Our calculations in sections 2 and 3 are tested by matlab in the cases $g=2$ and $g=3$.

\section{Modular Connections}
In this section we first recall the definition of connections in differential geometry. Then we introduce
the notion of modular connection on a Siegel upper half plane, and show several lemmas about it.
\subsection{Connections in differential geometry} For the backgrounds and notations on differential
geometry, especially on connections, we refer to the books
\cite{Chern} and \cite{Jost}. Here we just recall some basic definitions
and results on connections. Suppose $E$ is a $q$-dimensional real
vector bundle on a smooth manifold $M$, and $\Gamma(E)$ is the set
of smooth sections of $E$ on $M$. Let $T^*(M)$ be the cotangent space of $M$.
A connection on the vector bundle $E$ is a map
$$
D:\quad \Gamma(E)\longrightarrow\Gamma(T^*(M)\otimes E),
$$
which satisfies the following conditions
\begin{enumerate}
  \item For any $s_1,s_2\in\Gamma(E)$,
\[
D(s_1+s_2)=D(s_1)+D(s_2).
\]
\item For any $s\in\Gamma(E)$ and any $\alpha\in \mathbb C^\infty(M)$,
\[
D(\alpha s)=d\alpha\otimes\alpha D(s).
\]
\end{enumerate}
If $M$ has a generalized Riemannian metric $G=\sum_{i,j}g_{ij}du^idu^j$, by
the fundamental theorem of Riemannian geometry, $M$ has a unique
torsion-free and metric-compatible connection, called Levi-Civita
connection of $M$. The coefficients $\Gamma_{ij}^k$ of the
Levi-Civita connection are given by
\begin{equation}\label{formula}
  \Gamma_{ij}^k = \frac 12\sum_l g^{kl}\left(\frac{\partial g_{il}}{\partial
u^j}+\frac{\partial g_{jl}}{\partial u^i}- \frac{\partial
g_{ij}}{\partial u^l}\right),
\end{equation}
where $g^{ij}$ are elements of the matrix $(g^{ij}):=(g_{ij})^{-1}$.

 The following lemma is useful in the application of connections to automorphic forms.

\begin{lem}\label{lemma:basic} Let $\Gamma$ be a group, $(M,G)$ a Riemannian manifold and $D$ the
Levi-Civita connection on $M$.
If $\Gamma$ has a smooth left action on $M$ such that $G(\sigma_\star
X,\sigma_\star Y)=G(X,Y)$ for all $\sigma\in \Gamma,X,Y\in T(M)$,
 then
 \[
 \sigma D=D\sigma   \qquad (\sigma\in \Gamma).
 \]
Moreover, if $M$ is a complex manifold such that $\Gamma$ maps
$(r,s)$ forms to $(r,s)$ forms, and put $D=D^{1,0}+D^{0,1}$, where
$D^{1,0}$ is the holomorphic part, then for $\sigma\in \Gamma$
\[
\sigma D^{1,0}=D^{1,0}\sigma\quad\text{ and }\quad  \sigma D^{0,1}=D^{0,1}\sigma.
\]
\end{lem}
\begin{proof} $D$ is the unique torsion free connection which
preserves the Riemannian metric $G$. Since $G$ is $\Gamma$-invariant,
the connection $\sigma^{-1}D\sigma$ also preserves the
Riemannian metric and is torsion free for any $\sigma\in\Gamma$,
hence $\sigma D=D\sigma$. For more detail, see (\cite{Paradan}, P35).
\end{proof}

\subsection{Seigel upper half plane} We first fix some notations.
The Siegel upper half plane of degree $g\geq 1$ is defined to be the
$g(g+1)/2$ dimensional open complex variety
\[
\mathbb H_g:=\{Z=X+\sqrt{-1}Y\in M(g,\mathbb C)\mid Z^t=Z, Y>0\}.
\]
Write $Z=(Z_{ij})$. Set $\Omega=\{(i,j)\mid 1\le i\le j\le g\}$ with the dictionary order.
If $I=(i,j)\in\Omega$, we define $Z_I:=Z_{ij}$. Fix a series of coordinates
$\{dZ_I,d\bar{Z}_I\mid I\in\Omega\}$ on
$\mathbb H_g$. The symplectic group of degree $g>0$ over $\mathbb R$ is the group
\[
\Sp(2g,\mathbb R)=\left\{M\in GL(2g,\mathbb R)
\,\middle |\,  MJM^t=J\right\},
\]
where $J=\left(\begin{array}{cc} 0 & I_{g} \\ -I_{g} & 0
\end{array}\right)$. We usually write an element of $\Sp(2g,\mathbb R)$ in the form
$\left(\begin{array}{cc} A & B \\ C & D \end{array}\right)$, where $A,B,C$ and $D$ are $g\times g$ blocks.
The symplectic group $\Sp(2g,\mathbb R)$ acts on $\mathbb H_g$ by the rule:
\[
\gamma(Z):=(AZ+B)(CZ+D)^{-1}, \quad Z\in\mathbb H_g,\quad
\gamma=\left(\begin{array}{cc} A & B \\ C & D \end{array}\right)\in \Sp(2g,\mathbb R).
\]
By Maass (\cite{Maass1}, P98), $d(\gamma
Z)=(ZC^t+D^t)^{-1}dZ(CZ+D)^{-1}:=(d\tilde Z_{ij})$. Let
\[
(d\tilde Z_{11},d\tilde Z_{12},\cdots,d\tilde
Z_{1g},\cdots,\cdots,d\tilde Z_{gg}) =
(dZ_{11},dZ_{12},\cdots,dZ_{1g},\cdots,\cdots,dZ_{gg})\cdot S(\gamma,Z)
\]
where $S:=S(\gamma,Z)$ is a $\frac{g(g+1)}{2}\times \frac{g(g+1)}{2}$ matrix of holomorphic functions on
$\Sp(2g,\mathbb{Z})\times\mathbb H_g$.

From Lemma \ref{lemma:basic}, one can see that the connection matrix $\omega$ consisting of
the connection coefficients of the
Levi-Civita connection associated to the invariant metric $ds^2=\Tr(Y^{-1}dZ\cdot Y^{-1}d\bar Z)$ given by
Siegel (\cite{Maass1}, P8) on the Siegel upper plane $\mathbb H_g$ satisfies
\[
\gamma(\omega)=-S^{-1}\cdot dS +S^{-1}\cdot\omega\cdot S
\]
for all $\gamma\in\Sp(2g,\mathbb R)$. Refer also to the proof of Lemma \ref{lemma:modularconnectionBasic1} below.
But in the studying of modular forms, we only need that the equality holds for all $\gamma\in\Sp(2g,\mathbb Z)$,
the Siegel modular group.
So we need to introduce a weaker notion to study modular forms.

Now we recall the definition of Siegel modular forms, for more details, see \cite{Andrianov} and \cite{Geer}.
\begin{definition}
A (classical) Siegel modular form of weight $k$ (and degree $g$) is a holomorphic function
$f:\mathbb H_g\rightarrow \mathbb C$ such that
\[f(\gamma(Z))=\det(CZ+D)^k f(Z)\]
for all $\gamma=\left(
                  \begin{array}{cc}
                    A & B \\
                    C & D \\
                  \end{array}
                \right)\in \Sp(2g,\mathbb Z)
$ (with the usual holomorphicity requirement at $\infty$ when $g=1$).
\end{definition}
\subsection{Modular connections} The notations are the same as those above.
\begin{definition}[Modular Connection Coefficients (MCC)]\label{defi:MCC} The modular connection coefficient on
$\mathbb H_g$
is a series of $\mathbb C^\infty$-functions $\{\Gamma_{IJ}^K\mid
I,J,K\in\Omega\}$ such that for all $\gamma\in\Sp(2g,\mathbb Z)$,
\[
\gamma(\omega)=-S^{-1}\cdot dS +S^{-1}\cdot\omega\cdot S
,\quad\text{or}\quad S\cdot \gamma(\omega)=\omega\cdot S -dS,
\]
where $\omega=(\omega_I^J)$ and
$\omega_I^J=\sum_{K\in\Omega}\Gamma_{IK}^JdZ_K$. Here $I$ and $J$
are the row and column indices respectively. When
$\{\Gamma_{IJ}^K\}$ are holomorphic, we call it holomorphic MCC (HMCC). The matrix $\omega$ is called the
modular connection matrix.
\end{definition}

In the following, $\mathbb C^\infty(\mathbb H_g)$ is the set of $\mathbb C^\infty$ functions on $\mathbb H_g$,
and $\text{Hol}(\mathbb H_g)$ is the set of
holomorphic functions on $\mathbb H_g$.
\begin{definition}[Modular Connection] Let $\{\Gamma_{IJ}^K\}$ be a MCC (resp. HMCC) on $\mathbb H_g$ and $\Omega^\infty$ be
the commutative $\mathbb C^\infty(\mathbb H_g)$-algebra (resp. $\text{Hol}(\mathbb H_g)$-algebra) generated by
$\{dZ_I\}_{I\in\Omega}$ with the relations $dZ_IdZ_J=dZ_JdZ_I$ for any $I,J\in\Omega$. The linear operator
\[
D: \Omega^\infty\longrightarrow\Omega^\infty
\]
is uniquely defined by the following two relations
\[
D(dZ_K)=-\sum_{I,J\in\Omega}\Gamma_{IJ}^K dZ_IdZ_J
\]
and
\[
D(fdZ_{K_1}dZ_{K_2}\cdots dZ_{K_r})= df\cdot dZ_{K_1}\cdots dZ_{K_r}
+ \sum_{i=1}^r fdZ_{K_1}dZ_{K_2}\cdots D(dZ_{K_i})\cdots dZ_{K_r},
\]
and we call it the modular connection associated to $\{\Gamma_{IJ}^K\}$.
\end{definition}
One can easily show that $D(dZ_K)=-\sum_{I\in\Omega}\omega_I^K\cdot
dZ_I$ and
\[
(D(dZ_{11}),D(dZ_{12}),\cdots,\cdots,\cdots,D(dZ_{gg}))=-(dZ_{11},dZ_{12},\cdots,\cdots,\cdots,dZ_{gg})\cdot\omega
\]
 When $\{\Gamma_{IJ}^K\}$ is holomorphic, we also call $D$ a
holomorphic modular connection. Compared with the definition of
connections in differential geometry, except for the weaker conditions, the modular connection
also ignore the part on $\{d\bar{Z}_I\}_{I\in\Omega}$.

\subsection{Basic lemmas on modular connections} The following two lemmas are basic to our application of
modular connections to the Siegel modular forms. For the modular connections, we have the similar result to
Lemma \ref{lemma:basic}.
\begin{lem}\label{lemma:modularconnectionBasic1} Let $D$ be a modular connection on $\mathbb H_g$. Then
$\gamma D=D\gamma$ for any $\gamma\in\Sp(2g,\mathbb Z)$. Moreover, if $f$
is a Siegel modular forms of weight $2k$,
then $D\left(f(\det(dZ))^k\right)$ is invariant under the action of $\Gamma_g=\Sp(2g,\mathbb Z)$.
\end{lem}
\begin{proof}
Let $\alpha=
(dZ_{11},\cdots,dZ_{1g},dZ_{22},\cdots,dZ_{2g},\cdots,dZ_{gg})$.
Then $D\alpha=-\alpha\omega$. On one side,
\[
\gamma(D\alpha)= -\gamma(\alpha)\gamma(\omega)=-\alpha S (-S^{-1}dS\cdot +S^{-1}\cdot\omega\cdot S)
 =  \alpha(d S- \omega\cdot S).
\]
On the other side,
\[
D(\gamma\alpha)=D(\alpha\cdot S)= D(\alpha)\cdot S+ \alpha\cdot
dS=-\alpha\omega\cdot S+\alpha dS=\alpha(- \omega\cdot S+dS).
\]

For $f\in M_{2k}(\Gamma_g)$, we see $f(\det(dZ))^k$ is invariant under the action of $\Gamma_g$, and so is $D(f(\det(dZ))^k)$.
\end{proof}
The following lemma directly from Lemma \ref{lemma:basic} gives a modular connection.
\begin{lem}\label{lemma:1.8} The holomorphic part $D^{1,0}$ of the Levi-Civita connection
associated to the invariant metric $ds^2=\Tr(Y^{-1}dZ\cdot Y^{-1}d\bar Z)$ on the Siegel upper plane
$\mathbb H_g$ is a modular connection.
\end{lem}
We will give the explicit expression of $D^{1,0}(f\det(dZ)^k)$ in Proposition \ref{prop:DfdZK}.

\subsection{Modular connections in the classical case}
Let us consider the classical upper half plane $\mathbb H=\mathbb H_1$
to look for what condition of the coefficient $\Gamma:=\Gamma_{(1,1),(1,1)}^{(1,1)}$
of a modular connection should satisfy. Let $\omega=\Gamma dz$. One can easily check that for
$\gamma\in\SL(2,\mathbb Z)$
\[\gamma(\omega)=-S^{-1}dS+S^{-1}\omega S\Longleftrightarrow \frac{\gamma(\Gamma)}{(cz+d)^2}
=\Gamma+\frac{2c}{cz+d}.\]
Recall that (\cite{Koblitz}, P113)
\[
\frac 1{(cz+d)^2}\cdot\frac{\sqrt{-1}}{\text{Im}(\gamma z)}=\frac{\sqrt{-1}}{\text{Im}(z)}
+\frac{2c}{cz+d}\quad\text{and}\quad
\frac {\sqrt{-1}G_2(\gamma z)}{(cz+d)^2}=\sqrt{-1}G_2(z)+\frac{2c}{cz+d},
\]
where
\[
G_2(z)=\frac{1}{2\pi}\left(\sum_{n\not=0}\frac 1{n^2}+\sum_{m\not=0}\sum_{n\in\mathbb Z}\frac 1{(mz+n)^2}\right).
\]
So $\frac{\sqrt{-1}}y$ and $\sqrt{-1}G_2(z)$ give us two modular
connections on $\mathbb H$, which we denote by
$D_1$ and $D_2$ respectively. The later is holomorphic. We have
\[
D_1(fdz)=\left(\frac{df}{dz}-\frac{\sqrt{-1}}y
f\right)dzdz\quad\text{and}\quad
D_2(fdz)=\left(\frac{df}{dz}-\sqrt{-1}G_2(z)f\right)dzdz.
\]
By the expressions of $D_1(f(dz)^k)$ and $D_2(f(dz)^k)$ and by Lemma \ref{lemma:modularconnectionBasic1}, we have
\begin{cor} Let $f$ be a modular form of weight $2k$. Then $\frac{df}{dz}-\frac{\sqrt{-1}k}y f$
and $\frac{df}{dz}-\sqrt{-1}kG_2(z)f$ are modular forms of weight $2k+2$. They are
non-holomorphic and holomorphic, respectively.
\end{cor}
In fact, the modular connection $D_1$ comes from the Levi-Civita connection $D$ associated to the invariant metric
$ds^2=\frac{dz\,d\bar z}{y^2}$.
\begin{lem}
$D(dz)=-\frac{\sqrt{-1}}ydz\, dz \quad\text{and}\quad D(d\bar
z)=\frac{\sqrt{-1}}yd\bar z\, d\bar z$. So the coefficients of $D$ give the
modular connection $D_1$.
\end{lem}
\begin{proof}
Since $ds^2=\frac{dx^2+dy^2}{y^2}=\frac{dz\,d\bar z}{y^2}$, using
the coordinates ${dz,d\overline{z}}$ and the equality (\ref{formula}),
we have $\Gamma_{1,1}^{1}=\frac{\sqrt{-1}}{y},\
\Gamma_{2,1}^{1}=\Gamma_{1,2}^{1}=0, \
\Gamma_{2,1}^{2}=\Gamma_{1,2}^{2}=0$ and
$\Gamma_{2,2}^{2}=-\frac{\sqrt{-1}}{y}$.
\end{proof}
The connection $D_2$ is not from differential geometry. It is unique.
\begin{lem}[Uniqueness Lemma]\label{lemma:uniqe1} $\sqrt{-1}G_2(z)$ is the unique holomorphic function
$\Gamma$ satisfying
\[
\frac{\gamma(\Gamma)}{(cz+d)^2}=\Gamma+\frac{2c}{cz+d}~~\text{for all}~~\gamma=\left(
                                                                  \begin{array}{cc}
                                                                    a & b \\
                                                                    c & d \\
                                                                  \end{array}
                                                                \right)\in\SL(2,\mathbb Z),
\]
and so $D_2$ is the unique holomorphic modular connection on $\mathbb H$.
\end{lem}
\begin{proof}
We have
$\gamma(\sqrt{-1}G_2(z)-\Gamma)=(cz+d)^2(\sqrt{-1}G_2(z)-\Gamma)$
for all $\gamma\in\SL(2,\mathbb Z)$. So
$\sqrt{-1}G_2(z)-\Gamma$ is a modular form of weight 2 and
hence must be zero (\cite{Koblitz}, P117).
\end{proof}
We will generalize these results to $\mathbb H_g$.

\section{Derivative Operators of Siegel Modular Forms}
In this section, we first state the result determining the coefficients of
the invariant Levi-Civita connection on a Siegel upper half plane,
whose proof we put in the last section, then we compute the
expressions of the differential forms under this connection. Finally
we get a non-holomorphic derivative operator and a possible holomorphic
derivative operator.

\subsection{Coefficients of Levi-Civita connection}
The notations are the same as those in section 1. For
$I=(i,j)\in\Omega$, put
\[
 N(I)=\frac{(i-1)(2g-i)}2+j,
\]
which gives a one to one and order keeping correspondence between $\Omega$ and
$\{1,2,\cdots, \frac{g(g+1)}2\}$.
Write $u^{N(I)}=Z_{ij}$ and $u^{N(g,g)+N(i,j)}=\overline{Z}_{ij}$.
For $Z=X+\sqrt{-1}Y\in\mathbb H_g$, let $R=(R_{ij})=Y^{-1}$. For $1\le s\le g$, we set
\[
\Omega_s=\{(1,s),(2,s),\cdots,(s,s),(s,s+1),\cdots,(s,g)\}\subset\Omega.
\]
Let $K=(r,s)\in\Omega$. Assume that the elements of $Z$ in the column including $Z_K=Z_{rs}$ are
\[
u^{a_1}=Z_{1s},u^{a_2}=Z_{2s},\cdots,u^{a_s}=Z_{ss}, u^{a_{s+1}}=Z_{s+1,s}=Z_{s,s+1},\cdots,u^{a_g}=Z_{gs}=Z_{sg},
\]
and the elements in the row including $Z_{rs}$ are
\[
u^{b_1}=Z_{r1}=Z_{1r},u^{b_2}=Z_{r2}=Z_{2r},\cdots,u^{b_r}=Z_{rr}, u^{b_{r+1}}=Z_{r,r+1},\cdots,u^{b_g}=Z_{rg}.
\]
For $I\times J\in\Omega_s\times\Omega_r$, assume $Z_I=u^{a_i}$ and $Z_J=u^{b_j}$. Similarly do it for
$J\times I\in\Omega_s\times\Omega_r$. We define
\begin{equation}\label{coeff}
\Gamma_{IJ}^{K}\,(=\Gamma_{JI}^{K}):=\left\{\begin{array}{ll}
 \frac{\sqrt{-1}R_{ij}}{2^{(1-\delta(r,s))(1-\delta(a_i,b_j))}}&\mbox{ if }I\times J\text{ or }J\times I\in \Omega_s\times\Omega_r\\
0 &\text{ if }I\times J\text{ and }J\times I \not\in \Omega_s\times\Omega_r,
\end{array}\right.
\end{equation}
where $\delta(r,s)$ denotes the Kronecker delta symbol. For example, $\Gamma_{(1,i)(1,j)}^{(1,1)}=\sqrt{-1}R_{ij}$
and $\Gamma_{IJ}^{(1,1)}=0$ if $I$ or $J\not\in\Omega_1$.

Notice that if we use the coordinates $\{u^{a_i},u^{b_j}\}$, then the coefficients of the Levi-Civita connection satisfy
\[
{}_Z\!\Gamma_{IJ}^K={}_u\!\Gamma_{N(I)N(J)}^{N(K)}.
\]
In this paper we will use these two kinds of coordinates alternately.
The proof of the following lemma is long and complicated. For the
convenience of the reader, we put it in the last section.
\begin{lem}\label{lemma:mainLemma}
The coefficients $\{\Gamma_{IJ}^K\}$ defined in the equality (2.1) give the Levi-Civita connection
on $\mathbb H_g$ associated to the invariant metric
$ds^2=\Tr(Y^{-1}dZ\,Y^{-1}d\bar Z)$, and hence give a modular
connection, which we denote by $D$.
\end{lem}
\subsection{Expression of differential forms under $D$}
We first compute $D(dZ_K)$.
\begin{lem}\label{Lemma3.2} Let $K=(r,s)\in\Omega$. We have
\[
D(dZ_K)=-\sqrt{-1}(dZ_{s1},dZ_{s2},\cdots,dZ_{sg})Y^{-1}\cdot(dZ_{r1},dZ_{r2},\cdots,dZ_{rg})^t.
\]
\end{lem}
\begin{proof} The notations are as above.
If $r=s$, then $\Omega_r=\Omega_s$. By Lemma \ref{lemma:mainLemma}, we have
\[
\Gamma_{IJ}^{K}=\Gamma_{JI}^{K}=\left\{\begin{array}{ll}
 \sqrt{-1}R_{ij}&\mbox{ if }I\text{ and }J\in \Omega_r\\
0 &\text{ if }I\text{ or }J\not\in \Omega_r,
\end{array}\right.
\]
Hence,
\begin{eqnarray}
\nonumber D(dZ_K)&=& -\sum_{I,J\in\Omega}\Gamma_{I,J}^K dZ_I dZ_J =-\sum_{I,J\in\Omega_r}
\Gamma_{I,J}^K dZ_I dZ_J=-\sum_{i,j=1}^g\sqrt{-1}R_{ij}dZ_{si}dZ_{rj} \\
\nonumber &=&-\sqrt{-1}(dZ_{s1},dZ_{s2},\cdots,dZ_{sg})Y^{-1}\cdot(dZ_{r1},dZ_{r2},\cdots,dZ_{rg})^t.
\end{eqnarray}

If $r\not=s$, we assume $r< s$. Then $\Omega_r\bigcap \Omega_s=\{(r,s)\}$ and
$(\Omega_r\times \Omega_s)\bigcap (\Omega_s\times \Omega_r)=\{(r,s)\times (r,s)\}$. Put
$A=(\Omega_r\times \Omega_s)\bigcup (\Omega_s\times \Omega_r)$ and $B=(\Omega_r\times \Omega_s)\bigcap (\Omega_s\times \Omega_r)$.
We have again by Lemma \ref{lemma:mainLemma},
\[
\Gamma_{IJ}^{K}=\Gamma_{JI}^{K}=\left\{\begin{array}{lll}
 \frac{\sqrt{-1}R_{ij}}{2^{1-\delta(a_i,b_j)}}=\frac{\sqrt{-1}R_{ij}}{2}&\mbox{ if }I\times J\in A\setminus B\\\
\frac{\sqrt{-1}R_{ij}}{2^{1-\delta(a_i,b_j)}}=\sqrt{-1}R_{ij}&\mbox{ if }I\times J\in B .\\\
0 &\text{ if }I\times J\not\in A
\end{array}\right.
\]
Hence,
\begin{eqnarray}
\nonumber& &D(dZ_K)= -\sum_{I,J\in\Omega}\Gamma_{I,J}^K dZ_I dZ_J =-\sum_{I\times J\in A}\Gamma_{I,J}^K dZ_I dZ_J \\
\nonumber &=&-2\sum_{I\in \Omega_r \atop J\in \Omega_s}\Gamma_{I,J}^K dZ_I dZ_J+\sum_{I\in
\Omega_r\bigcap\Omega_s\atop J\in \Omega_r\bigcap\Omega_s}\Gamma_{I,J}^K dZ_I dZ_J
=-2\sum_{I\times J\in \Omega_r \times \Omega_s-B}\Gamma_{I,J}^K dZ_I dZ_J-\sum_{I\times J\in B}\Gamma_{I,J}^K dZ_I dZ_J\\
\nonumber &=&-\sum_{i,j=1}^g\sqrt{-1}R_{ij}dZ_{si}dZ_{rj}
=-\sqrt{-1}(dZ_{s1},dZ_{s2},\cdots,dZ_{sg})Y^{-1}\cdot(dZ_{r1},dZ_{r2},\cdots,dZ_{rg})^t.
\end{eqnarray}

The case $s<r$ is similar.
\end{proof}
\begin{prop}\label{prop:Ddz}
$D(\det(dZ))=-\sqrt{-1}\Tr(Y^{-1}dZ)\det(dZ)$.
\end{prop}
\begin{proof}
Put $\alpha_i=(dZ_{i1},dZ_{i2},\cdots,dZ_{ig})$ and $\beta_j=(dZ_{1j},dZ_{2j},\cdots,dZ_{gj})^t$. By
Lemma \ref{Lemma3.2},
$D(d(Z_{ij}))=-\sqrt{-1}\alpha_i Y^{-1}\beta_j$ and $D(\alpha_i)=-\sqrt{-1}\alpha_i Y^{-1}dZ$. Thus
\[
D(\det(dZ))=-\sqrt{-1}\left\{\det\left(\begin{array}{c}\alpha_1Y^{-1}dZ\\ \alpha_2\\ \vdots\\ \alpha_g
\end{array}\right)+\det\left(\begin{array}{c}\alpha_1\\ \alpha_2Y^{-1}dZ\\ \vdots \\ \alpha_g
\end{array}\right)+\cdots+\det\left(\begin{array}{c}\alpha_1\\ \vdots \\ \alpha_{g-1}\\ \alpha_gY^{-1}dZ
\end{array}\right)\right\}.
\]
Let $A=Y^{-1}dZ=(A_{ij})$ and $dZ[i,j]$ be the algebraic cofactor of
$dZ$ at the position $(i,j)$. By the formula above, we have
\begin{eqnarray*}
& &\sqrt{-1}D(\det(dZ))=\sum_{k,j,i=1}^ndZ_{ki}\cdot A_{ij}\cdot dZ[k,j]
=\sum_{j,i=1}^nA_{ij}\sum_{k=1}^ndZ_{ik}\cdot dZ[k,j]\\
&=&\sum_{j,i=1}^nA_{ij}\delta(i,j)\det(dZ)=\Tr(A)\det(dZ)=\Tr(Y^{-1}dZ)\det(dZ).
\end{eqnarray*}
\end{proof}
Put
\[
\frac\partial{\partial Z}=(\partial_{ij})_{g\times g},\qquad
\partial_{ij}=\frac{1+\delta(i,j)}2\cdot\frac\partial{\partial Z_{ij}}
\]
as in the introduction. Then for any $\mathbb C^\infty$-function $f$ on $\mathbb
H_g$,
\[df=\sum_{1\leq i\leq j\leq g}\frac{\partial f}{\partial Z_{ij}} dZ_{ij}
=\sum_{i=1}^g\sum_{j=1}^g \partial_{ij}f dZ_{ij}=
\Tr\left(\frac{\partial}{\partial Z}f\cdot dZ\right) \]

\begin{prop}\label{prop:DfdZK}
For any $\mathbb C^\infty$-function $f$ on $\mathbb H_g$,  we have
\[D\left(f\det(dZ)^k\right)=\Tr\left(\left[\frac{\partial}{\partial
Z}-\sqrt{-1}kY^{-1}\right]f dZ\right)\det(dZ)^k.\]
\end{prop}
\begin{proof} Since $df=\Tr(\frac\partial{\partial Z}f\cdot dZ)$,
we have, by Proposition \ref{prop:Ddz},
\begin{eqnarray*}
D\left(f\det(dZ)^k\right)&=& df\cdot\det(dZ)^k+f\cdot D((\det(dZ)^k)=(df-\sqrt{-1}k f\Tr(Y^{-1}dZ))\det(dZ)^k\\
&=& \Tr\left(\left[\frac{\partial}{\partial
Z}-\sqrt{-1}kY^{-1}\right]f dZ\right)\det(dZ)^k.
\end{eqnarray*}
\end{proof}

In the following, $\Omega^i_{\mathbb H_g}$ is the sheaf of holomorphic $i$-forms on $\mathbb H_g$.
Recall that a section of $\Omega^1_{\mathbb H_g}$ can be written as $\Tr(GdZ)$, where $G$ is a symmetric matrix
of holomorphic functions on $\mathbb H_g$.
\begin{prop}\label{prop:last}
For any section $\Tr(GdZ)\in\Omega_{\mathbb H_g}^1$, where $G$ is a
symmetric matrix of holomorphic functions on $\mathbb H_g$, we have
\[
D(\Tr(GdZ))=\Tr\left\{\left[\left(\frac\partial{\partial Z}\right)^t\otimes G\right]\cdot[dZ\otimes dZ]\right\}
-\sqrt{-1}\Tr\left(GdZ\cdot Y^{-1}dZ\right),
\]
where $\otimes$ is the Kronecker product of matrices.
\end{prop}
To show this proposition, we need the following lemma.
\begin{lem} We have

\begin{enumerate}
\item Let $A=(a_{ij})_{n\times n},B=(b_{ij})_{n\times n},C=(c_{ij})_{n\times n},D=(d_{ij})_{n\times n}$. Then
\[
\Tr((A\otimes B)(C\otimes D))=\sum_{i,j,k,l=1}^na_{ij}b_{kl}c_{lk}d_{ji}=\Tr((A\otimes C)(B\otimes D)),
\]
\item $d(\Tr(GdZ))=\Tr\left(\left(\left(\frac\partial{\partial Z}\right)^t\otimes G\right)\cdot(dZ\otimes dZ)\right)$
for a symmetric matrix $G=(G_{ij})$ of functions.
\end{enumerate}
\end{lem}
\begin{proof} The proof of (1) is easy. We only show (2). By (1),
\begin{eqnarray*}
d(\Tr(G dZ))&=& d\left(\sum_{i,j=1}^gG_{ij}dZ_{ij}\right)=\sum_{i,j=1}^g\sum_{k,l=1}^g\partial_{kl}G_{ij}\cdot
dZ_{kl}dZ_{ji}\\
&=& \Tr\bigg(\big((\partial_{kl})^t\otimes (G_{ij})\big)\cdot(dZ\otimes dZ)\bigg).
\end{eqnarray*}
\end{proof}
\begin{proof}[Proof of Proposition \ref{prop:last}] As
$dZ=(\alpha_1^t,\alpha_2^t,\cdots,\alpha_g^t)^t=(\beta_1,\beta_2,\cdots,\beta_g)$, we have
\begin{eqnarray*}
& & D(\Tr(G dZ))-\Tr\left\{\left[\left(\frac\partial{\partial Z}\right)^t\otimes G\right]\cdot[dZ\otimes dZ]\right\}
=\Tr\left(G D(dZ)\right)\\
&=&\sum_{i,j=1}^gG_{ij}D(dZ_{ij})=-\sqrt{-1}\sum_{i,j=1}^gG_{ij}d(\alpha_i)Y^{-1}d(\beta_j)
=-\sqrt{-1}\Tr\left(GdZY^{-1}dZ\right).
\end{eqnarray*}
\end{proof}

\subsection{Derivative operators}
If $0\neq f\in M_{2k}(\Gamma_g)$, we have, by Proposition \ref{prop:DfdZK} and
Lemma \ref{lemma:modularconnectionBasic1},
\[
\Tr\left(\left[\frac{\partial}{\partial Z}-\sqrt{-1}kY^{-1}\right]f dZ\right)\det(dZ)^k
\]
is invariant under the action of $\Gamma_g$. Let $\gamma=\left(\begin{array}{cc} A
& B\\ C& D\end{array}\right)\in\Gamma_g$. Since
$\gamma(\det(dZ))^k=\frac{\det(dZ)^k}{\det(CZ+D)^{2k}}$ and
$\gamma(f)=\det(CZ+D)^{2k}f$, we have
\[
\Tr\left(\frac 1f\left[\frac{\partial}{\partial Z}-\sqrt{-1}kY^{-1}\right]fdZ\right)
\]
is invariant under $\Gamma_g$.
Put $h:=\frac 1f\left(\frac{\partial}{\partial Z}-\sqrt{-1}kY^{-1}\right)f$.
Then $h$ is a symmetric matrix of functions on $\mathbb H_g$, and $\Tr(hdZ)$ is invariant under the action of $\Gamma_g$.
Since for any $\gamma=\left(
\begin{array}{cc}
 A & B \\
   C & D \\
   \end{array}
    \right)
\in\Gamma_g$, $d(\gamma Z)=(ZC^t+D^t)^{-1}\cdot dZ\cdot(CZ+D)^{-1}$, we
have (see (\cite{Geer}, P210))
\[h(\gamma Z)=(CZ+D)h(Z)(ZC^t+D^t).\]
Thus $\det(h)=\det\left(\frac 1f\left(\frac{\partial}{\partial Z}-\sqrt{-1}kY^{-1}\right)f \right)$
is a non-holomorphic
Siegel modular form of weight 2. Let $\widetilde{M}_{k}(\Gamma_g)$ be the $\mathbb C^\infty$-Siegel
modular forms of weight $k$ as in the introduction. Finally we get
\begin{theorem}\label{mainthm}
If $f\in M_{2k}(\Gamma_g)$, then
$\det\left(\left(\frac{\partial}{\partial
Z}-\sqrt{-1}kY^{-1}\right)f \right)\in
\widetilde{M}_{2kg+2}(\Gamma_g)$.
\end{theorem}

Let $f\in M_{2r}(\Gamma_g)$ and $h\in M_{2s}(\Gamma_g)$. We have
\begin{eqnarray*}
& & D(f\det (dZ)^r)\cdot h\det (dZ)^s-f\det (dZ)^r\cdot D(h\det (dZ)^s)\\
&=&\Tr\left(\left[h\frac{\partial}{\partial Z}f-f\frac{\partial}{\partial Z}h\right]dZ\right)\det (dZ)^{r+s}
\end{eqnarray*}
is invariant under $\Gamma_g$. The same consideration as above gives us
\[
\det\left(h\frac{\partial}{\partial Z}f-f\frac{\partial}{\partial Z}h\right)\in M_{2(r+s)g+2}.
\]
We can continue this construction to find combinations of higher derivatives of $f$ and $h$ which are modular.
By setting $[f,h]_0:=fh$,
$[f,h]_1:=\det(h\frac{\partial}{\partial Z}f-f\frac{\partial}{\partial Z}h)$, and so on, one would get the
Rankin-Cohen brackets.

\subsection{The unique theorem} We first show a lemma.
\begin{lem} Let $\gamma=\left(\begin{array}{cc} A & B\\ C& D\end{array}\right)\in\Sp(2g,\mathbb Z)$. Then
\[
(CZ+D)^{-1}\sqrt{-1}(\text{Im}\gamma(Z))^{-1}=\sqrt{-1}(\text{Im}(Z))^{-1}\cdot(CZ+D)^t+2C^t.
\]
\end{lem}
\begin{proof} Since Im$(\gamma(Z))=((C\bar Z+D)^t)^{-1}Y(CZ+D)^{-1}$ (see \cite{Andrianov}) and $Y^t=Y$, we have
\begin{eqnarray*}
(\text{Im}\gamma(Z))^{-1}&=&(CZ+D)Y^{-1}\cdot(C\bar Z+D)^t=(CZ+D)Y^{-1}\cdot(CZ+D-2\sqrt{-1}CY)^t\\
&=&(CZ+D)Y^{-1}\cdot(CZ+D)^t-(CZ+D)2\sqrt{-1}C^t,
\end{eqnarray*}
and thus the result.
\end{proof}
\begin{theorem}\label{mainthm1}
For any symmetric $g\times g$ matrix $G=(G_{ij})$ consisting of $\mathbb
C^\infty$ (or holomorphic) functions on $\mathbb H_g$ which satisfies the transformation
formula
\[
(CZ+D)^{-1}\gamma(G)=G\cdot(CZ+D)^t+2C^t,
\]
there exists a unique modular connection $\mathbb D$ such that
\[
\mathbb D(dZ_{rs})=-\sum_{i,j=1}^gG_{ij}dZ_{si}dZ_{rj}
\quad{and}\quad \mathbb
D(f(\det(dZ)^k)=\Tr\left(\left[\frac{\partial}{\partial
Z}-kG\right]f dZ\right)(\det(dZ))^k,
\]
and thus $G$ gives a derivative operator $M_{2k}\rightarrow\widetilde{M}_{2kg+2}$ by
$f\mapsto\det\left(\left[\frac{\partial}{\partial
Z}-kG\right]f\right)$.
Furthermore, there exists at most one holomorphic
symmetric matrix $G$ to satisfy the transformation formula. If such a $G$ exists,
the operator corresponding to $G$ is holomorphic.
\end{theorem}
\begin{proof} One notes that in the definition of modular connection coefficients $\Gamma_{IJ}^K$,
we only need the transformation law $\gamma(\omega)=-S^{-1}\cdot dS
+S^{-1}\cdot\omega\cdot S$ for $\gamma\in\Sp(2g,\mathbb Z)$. If $G$
has the same transformation law as $\sqrt{-1}(\text{Im}(Z))^{-1}$,
then we can use the same method in Lemma \ref{lemma:mainLemma} to
construct $\{\Gamma_{IJ}^K\}$ and to calculate the expressions of
the differential forms under $\mathbb D$. The same discussion as in subsection 2.3
tells us $\det\left(\left[\frac{\partial}{\partial
Z}-kG\right]f\right)\in\widetilde{M}_{2kg+2}$.

On the uniqueness, let $G$ and $\tilde G$ be two holomorphic
matrices to satisfy the transformation formula. Then
\[
(CZ+D)^{-1}(G(\gamma Z)-\tilde G(\gamma Z))=(G(Z)-\tilde
G(Z))(CZ+D)^t.
\]
So $\Tr\{(G-\tilde G)dZ\}\in(\Omega_{\mathbb H_g}^1)^{\Gamma_g}$. In
\cite{Weissauer1} and \cite{Weissauer2} R. Weissauer proved that if
$v$ is not of the form $[u]:=ug-\frac{1}{2}u(u-1)$, then
$(\Omega_{\mathbb H_g}^v)^{\Gamma_g}=0$. If $g\geq 2$, one gets
$(\Omega_{\mathbb H_g}^1)^{\Gamma_g}=0$, and thus $0=\Tr\{(G-\tilde
G)dZ\}=2\sum\limits_{i<j}(G_{ij}-\tilde G_{ij})dZ_{ji}
+\sum\limits_{i=1}^g (G_{ii}-\tilde G_{ii})dZ_{ii}$. Hence
$G_{ij}=\tilde G_{ij}$ for all $1\leq i,j\leq g$. The case $g=1$ has
been proved in Lemma \ref{lemma:uniqe1}.
\end{proof}
\begin{quest}
Does there exist such a $G$? If so, how to construct it?
\end{quest}

\section{Explicit Construction of the Levi-Civita Connection}
 In this section we give the proof of Lemma \ref{lemma:mainLemma}. We denote $I,J,K,L,\cdots$ the elements in
$\Omega$, $\ i,j,k,l,r,s,\cdots$ the elements in $\{1,2,\cdots,g\}$ and
$\alpha,\beta,\gamma,\delta,\epsilon$ the elements in $\{1,2,\cdots,g(g+1)/2\}$.

\subsection{Riemannian metric} Let $R:=(R_{ij})_{g\times g}=Y^{-1}$. Then
\begin{eqnarray*}
\nonumber  ds^2 &=& \Tr(R\cdot dZ\cdot R\cdot d\overline{Z})
=\sum_{i\leq j}\sum_{r\leq
s}2^{2-\delta(i,j)-\delta(r,s)}\times\frac{R_{ir}R_{js}+R_{jr}R_{is}}{2}dZ_{ij}d\overline{Z}_{rs},
\end{eqnarray*}
and thus the Riemannian metric matrix associated to
$ds^2=\Tr(Y^{-1}dZ\cdot Y^{-1}d\overline{Z})$ is given by
\[
G=\left(
  \begin{array}{cc}
    0 & W \\
    W & 0 \\
  \end{array}
\right)
\]
where
\[
W=(W_{IJ})_{I,J\in\Omega},\quad W_{IJ}=\frac{R_{ir}R_{js}+R_{jr}R_{is}}{2^{\delta(i,j)+\delta(r,s)}},~~
\text{ if }I=(i,j)\text{ and }J=(r,s).
\]
\begin{lem} The inverse $W^{-1}$ of $W$ is given by
\[
M:=(M_{IJ})_{I,J\in\Omega},\quad M_{IJ}=Y_{ir}Y_{js}+Y_{jr}Y_{is},~~\text{ if }I=(i,j)\text{ and }J=(r,s).
\]
\end{lem}
\begin{proof} We need to show that for any $I=(i,j)\in\Omega$ and $K=(p,q)\in\Omega$,
\[\sum_{J\in\Omega}M_{IJ}W_{JK}=\delta(I,K).\]
By direct computations, we have
\begin{eqnarray*}
 & & \sum_{1\leq r\leq s\leq g}M_{(i,j),(r,s)}W_{(r,s),(p,q)} =
 \sum_{1\leq r\leq s\leq g}(Y_{ir}Y_{js}+Y_{jr}Y_{is})\frac{R_{rp}R_{sq}+R_{sp}R_{rq}}{2^{\delta(r,s)+\delta(p,q)}} \\
 &=&\sum_{1\leq r\leq g}2^{1-\delta(p,q)}Y_{ir}Y_{jr}R_{rp}R_{rq}+\sum_{1\leq r< s\leq g}(Y_{ir}Y_{js}+Y_{jr}Y_{is})
 \frac{R_{rp}R_{sq}+R_{sp}R_{rq}}{2^{\delta(p,q)}}  \\
  &=& 2^{-\delta(p,q)}\sum_{1\leq r\leq g}\sum_{1\leq s\leq g}
  (Y_{ir}Y_{js}R_{rp}R_{sq}+Y_{jr}Y_{is}R_{rp}R_{sq})\\
 &=& 2^{-\delta(p,q)}\sum_{1\leq r\leq g}
  (Y_{ir}R_{rp}\delta(j,q)+Y_{jr}R_{rp}\delta(i,q))\\
  &=& 2^{-\delta(p,q)}\{\delta(i,p)\delta(j,q)+\delta(j,p)\delta(i,q)\}= \delta(I,K).
\end{eqnarray*}
Notice that, in the last three steps, we have used the equality
$\sum_{1\leq r\leq g}Y_{ir}R_{rp}=\delta(r,p)$, which comes from $R=Y^{-1}$.
\end{proof}

\subsection{Connection Coefficients} As before, we put
$u^{N(i,j)}=Z_{ij}$, $u^{N(g,g)+N(i,j)}=\overline{Z}_{ij}$,
\[G=\left(
      \begin{array}{cc}
        0 & W \\
        W & 0 \\
      \end{array}
    \right)\quad\text{and}\quad
    \widehat{G}:=G^{-1}=\left(
      \begin{array}{cc}
        0 & M \\
        M & 0 \\
      \end{array}
    \right).
\]
By the Equality (\ref{formula}), we have
\[
\Gamma_{\alpha,\beta}^{\gamma}=\sum_{1\leq \rho\leq g(g+1)}\frac{1}{2}\widehat{G}_{\gamma,\rho}
\left(\frac{\partial G_{\alpha,\rho}}{\partial
u^{\beta}}+\frac{\partial G_{\beta,\rho}}{\partial
u^{\alpha}}-\frac{\partial G_{\alpha,\beta}}{\partial
u^{\rho}}\right).
\]
Assume $0\leq \alpha,\beta,\gamma\leq\frac{g(g+1)}{2}$, then $G_{\alpha,\beta}=0$. We have:
\[
\Gamma_{\alpha,\beta}^{\gamma}=\sum_{1\leq \rho\leq g(g+1)}\frac{1}{2}\widehat{G}_{\gamma,\rho}
\left(\frac{\partial G_{\alpha,\rho}}{\partial
u^{\beta}}+\frac{\partial G_{\beta,\rho}}{\partial
u^{\alpha}}\right).
\]
Hence for $I,J,K\in\Omega$
\[
\Gamma_{I,J}^{K}=\sum_{L\in\Omega}\frac{1}{2}M_{K,L}
\left(\frac{\partial W_{I,L}}{\partial Z_{J}}+\frac{\partial
W_{J,L}}{\partial Z_{I}}\right).
\]
Again, notice that
\[
\sum_{L\in\Omega}M_{K,L}W_{I,L}=\delta(K,I)\ \text{ and }\ \sum_{L\in\Omega}M_{K,L}W_{J,L}=\delta(K,J).
\]
Do partial derivatives on both sides with respect to $Z_J$ and $Z_I$ respectively, we have
\[
\sum_{L\in\Omega}M_{K,L}\frac{\partial W_{I,L}}{\partial Z_J}
+\sum_{L\in\Omega}\frac{\partial M_{K,L}}{\partial Z_J}W_{I,L}= 0,
\]
and
\[
\sum_{L\in\Omega}M_{K,L}\frac{\partial W_{J,L}}{\partial
Z_I}+\sum_{L\in\Omega}\frac{\partial M_{K,L}}{\partial
Z_I}W_{J,L}= 0.
\]
Finally we get
\[
\Gamma_{I,J}^K=-\frac{1}{2}\left(\sum_{L\in\Omega}\frac{\partial M_{K,L}}{\partial Z_J}W_{I,L}
+\sum_{L\in\Omega}\frac{\partial M_{K,L}}{\partial
Z_I}W_{J,L}\right).
\]

If $I=(i,j),J=(r,s),K=(p,q)$ and $L=(a,b)\in\Omega$, then
$M_{I,J}=Y_{ir}Y_{js}+Y_{jr}Y_{is}$ and
\begin{eqnarray*}
& &\frac{\partial M_{K,L}}{\partial Z_J}=\frac{\partial (Y_{pa}Y_{qb}+Y_{qa}Y_{pb})}{\partial Z_J} \\
&=&-\frac{\sqrt{-1}}{2}\left\{\sigma_{(p,a),(r,s)}Y_{qb}+\sigma_{(q,b),(r,s)}Y_{pa}+\sigma_{(q,a),(r,s)}Y_{pb}
+\sigma_{(p,b),(r,s)}Y_{qa}\right\}
\end{eqnarray*}
Here we define:
\[
\sigma_{(p,a),(r,s)}=\left\{\begin{array}{ll}
1, &\mbox{ if } Z_{pa}=Z_{rs},\\
0, &\mbox{ if } Z_{pa}\neq Z_{rs}.
\end{array}\right.
\]

One should notice the difference of the notation above with the notation
$\delta_{(p,a),(r,s)}:=\delta((p,a),(r,s))=\delta(p,r)\delta(a,s)$. These two notations have the following
relations:
\[\sigma_{(p,a),(r,s)}=\delta_{(p,a),(r,s)}+\delta_{(p,a),(s,r)}-\delta_{(p,a),(r,s)}\cdot\delta_{(p,a),(s,r)}.\]
Using the equality
$W_{IJ}=\frac{R_{ir}R_{js}+R_{jr}R_{is}}{2^{\delta(i,j)+\delta(r,s)}}$
and others above, we have
\begin{eqnarray*}
& &\sum_{L\in\Omega}\frac{\partial M_{K,L}}{\partial Z_J}W_{I,L}= -\frac{\sqrt{-1}}{2}\sum_{L=(a,b)\in\Omega}
\{\sigma_{(p,a),(r,s)}Y_{qb}+\sigma_{(q,b),(r,s)}Y_{pa}\\
 & &+\sigma_{(q,a),(r,s)}Y_{pb}+\sigma_{(p,b),(r,s)}Y_{qa}\}W_{I,L}\\
&=&-\frac{\sqrt{-1}}{2^{1+\delta(i,j)}}\Bigg\{\sum_{1\leq a\leq g}\sigma_{(p,a),(r,s)}\delta(q,j)R_{ia}
+\sum_{1\leq b\leq g}\sigma_{(p,b),(r,s)}\delta(q,i)R_{jb}  \\
& &+\sum_{1\leq a\leq g}\sigma_{(q,a),(r,s)}\delta(p,i)R_{ja}
+\sum_{1\leq b\leq g}\sigma_{(q,b),(r,s)}\delta(p,j)R_{ib}\Bigg\}.
\end{eqnarray*}
While
\begin{eqnarray*}
\nonumber \sum_{1\leq a\leq g}\sigma_{(p,a),(r,s)}\delta(q,j)R_{ia}&=& \sum_{1\leq a\leq g}\{\delta(p,r)\delta(a,s)+\delta(p,s)\delta(a,r)
 \\
\nonumber & &-\delta(p,r)\delta(a,s)\delta(p,s)\delta(a,r)\}\delta(q,j)R_{ia}\\
 \nonumber &=&
\delta(q,j)\{\delta(p,r)R_{is}+\delta(p,s)R_{ir}-\delta(p,r)\delta(p,s)R_{is}\},\\
 \nonumber \sum_{1\leq b\leq g}\sigma_{(p,b),(r,s)}\delta(q,i)R_{jb}&=&\delta(q,i)
 \{\delta(p,r)R_{js}+\delta(p,s)R_{jr}-\delta(p,r)\delta(p,s)R_{js}\},\\
\nonumber \sum_{1\leq a\leq g}\sigma_{(q,a),(r,s)}\delta(p,i)R_{ja}&=&\delta(p,i)\{\delta(q,r)R_{js}+
\delta(q,s)R_{jr}-\delta(q,r)\delta(q,s)R_{js}\},\\
\nonumber \sum_{1\leq b\leq
g}\sigma_{(q,b),(r,s)}\delta(p,j)R_{ib}&=&\delta(p,j)\{\delta(q,r)R_{is}+\delta(q,s)R_{ir}-\delta(q,r)\delta(q,s)R_{is}\}.
\end{eqnarray*}
Combining these equalities together, we have
\begin{eqnarray*}
\nonumber \sum_{L\in\Omega}\frac{\partial M_{K,L}}{\partial
Z_J}W_{I,L}&=& -\frac{\sqrt{-1}}{2^{1+\delta(i,j)}}\{\delta(q,j)\{\delta(p,r)R_{is}+\delta(p,s)R_{ir}-\delta(p,r)\delta(p,s)R_{is}\} \\
\nonumber & &+\delta(q,i)\{\delta(p,r)R_{js}+\delta(p,s)R_{jr}-\delta(p,r)\delta(p,s)R_{js}\}  \\
\nonumber & &+\delta(p,i)\{\delta(q,r)R_{js}+\delta(q,s)R_{jr}-\delta(q,r)\delta(q,s)R_{js}\} \\
\nonumber & &+\delta(p,j)\{\delta(q,r)R_{is}+\delta(q,s)R_{ir}-\delta(q,r)\delta(q,s)R_{is}\}\} \\
\end{eqnarray*}
Similarly, we have
\begin{eqnarray*}
\nonumber \sum_{L\in\Omega}\frac{\partial M_{K,L}}{\partial
Z_I}W_{J,L}&=& -\frac{\sqrt{-1}}{2^{1+\delta(r,s)}}\{\delta(q,s)\{\delta(p,i)R_{rj}+\delta(p,j)R_{ri}-\delta(p,i)\delta(p,j)R_{rj}\} \\
\nonumber & &+\delta(q,r)\{\delta(p,i)R_{sj}+\delta(p,j)R_{si}-\delta(p,i)\delta(p,j)R_{sj}\}  \\
\nonumber & &+\delta(p,r)\{\delta(q,i)R_{sj}+\delta(q,j)R_{si}-\delta(q,i)\delta(q,j)R_{sj}\} \\
\nonumber &
&+\delta(p,s)\{\delta(q,i)R_{rj}+\delta(q,j)R_{ir}-\delta(q,i)\delta(q,j)R_{rj}\}\}
\end{eqnarray*}
Finally we get
\begin{lem}\label{lemma:last}
\begin{eqnarray*}
\nonumber \Gamma_{I,J}^K &=&
-\frac{1}{2}\left(\sum_{L\in\Omega}\frac{\partial M_{K,L}}{\partial
Z_J}W_{I,L} +\sum_{L\in\Omega}\frac{\partial M_{K,L}}{\partial
Z_I}W_{J,L}\right) \\
\nonumber &=& \frac{\sqrt{-1}}{2^{2+\delta(i,j)}}\{\delta(q,j)\{\delta(p,r)R_{is}+\delta(p,s)R_{ir}-\delta(p,r)\delta(p,s)R_{is}\} \\
\nonumber & &+\delta(q,i)\{\delta(p,r)R_{js}+\delta(p,s)R_{jr}-\delta(p,r)\delta(p,s)R_{js}\}  \\
\nonumber & &+\delta(p,i)\{\delta(q,r)R_{js}+\delta(q,s)R_{jr}-\delta(q,r)\delta(q,s)R_{js}\} \\
\nonumber & &+\delta(p,j)\{\delta(q,r)R_{is}+\delta(q,s)R_{ir}-\delta(q,r)\delta(q,s)R_{is}\}\} \\
\nonumber & &+  \frac{\sqrt{-1}}{2^{2+\delta(r,s)}}\{\delta(q,s)\{\delta(p,i)R_{rj}+\delta(p,j)R_{ri}-\delta(p,i)\delta(p,j)R_{rj}\} \\
\nonumber & &+\delta(q,r)\{\delta(p,i)R_{sj}+\delta(p,j)R_{si}-\delta(p,i)\delta(p,j)R_{sj}\}  \\
\nonumber & &+\delta(p,r)\{\delta(q,i)R_{sj}+\delta(q,j)R_{si}-\delta(q,i)\delta(q,j)R_{sj}\} \\
\nonumber
 &&+\delta(p,s)\{\delta(q,i)R_{rj}+\delta(q,j)R_{ir}-\delta(q,i)\delta(q,j)R_{rj}\}\}
\end{eqnarray*}
\end{lem}
Using the lemma \ref{lemma:last} above, one can easily show that in the case
$K=(p,q)=(1,1),I=(i,j)=(1,j),J=(r,s)=(1,s)$, we have
\begin{itemize}
  \item if $j=s=1$, then $\Gamma_{I,J}^K=\sqrt{-1}R_{1,1}$;
  \item if $j=1,s\neq 1$, then $\Gamma_{I,J}^K=\sqrt{-1}R_{1,s}$;
  \item if $j\neq 1,s\neq 1$, then $\Gamma_{I,J}^K=\sqrt{-1}R_{j,s}=\sqrt{-1}R_{s,j}$.
\end{itemize}
In general case, if $K=(p,q),I=(i,j),J=(r,s)$, and if both
$(Z_{ij},Z_{rs})$ and $(Z_{rs},Z_{ij})$ do not belong to
$\{Z_{1p},Z_{2p},\cdots,Z_{gp}\}\times\{Z_{1q},Z_{2q},\cdots,Z_{gq}\}$, then
all terms in the last equality of lemma \ref{lemma:last} are zero, hence
$\Gamma_{I,J}^K=\Gamma_{J,I}^K=0$.

If $p=q$, then
\begin{eqnarray*} \Gamma_{I,J}^{(p,p)}&=&\frac{\sqrt{-1}}{2^{1+\delta(i,j)}}\{\delta(p,i)\{\delta(p,r)R_{js}+
\delta(p,s)R_{jr}-\delta(p,r)\delta(p,s)R_{jr}\}  \\
\nonumber &&+\delta(p,j)\{\delta(p,r)R_{is}+\delta(p,s)R_{ir}-\delta(p,r)\delta(p,s)R_{ir}\}\} \\
\nonumber &&+\frac{\sqrt{-1}}{2^{1+\delta(r,s)}}\{\delta(p,r)\{\delta(p,i)R_{sj}+\delta(p,j)R_{si}-\delta(p,i)\delta(p,j)R_{sj}\}  \\
&&+\delta(p,s)\{\delta(p,i)R_{rj}+\delta(p,j)R_{ri}-\delta(p,i)\delta(p,j)R_{rj}\}\}.
\end{eqnarray*}
If $Z_{ij}=Z_{ji}$ and $Z_{rs}=Z_{sr}$ belong to the same row or
column with $Z_{pp}$, then $i=r=p$, or $i=s=p$, or $j=r=p$, or $j=s=p$.
\begin{itemize}
  \item If $i=r=p$, one can use the formula above to show that $\Gamma_{I,J}^K=\sqrt{-1}R_{js}$
  \item If $i=s=p$, then $\Gamma_{I,J}^K=\sqrt{-1}R_{jr}$.
  \item If $j=r=p$, then $\Gamma_{I,J}^K=\sqrt{-1}R_{is}$.
  \item If $j=s=p$, then $\Gamma_{I,J}^K=\sqrt{-1}R_{ir}$.
\end{itemize}

If $p<q$, we may assume that $Z_{ij}(i\leq j)$ belong to the same row with
$Z_{pq}$ and $Z_{rs}(r\leq s)$ belongs to the same column with
$Z_{pq}$ (Other cases can be proved in the same way). Then $i=p\leq j$, $r\leq s=q$ and
\begin{eqnarray*}
\nonumber \Gamma_{I,J}^{K}&=&\frac{\sqrt{-1}}{2^{2+\delta(i,j)}}(\delta(q,j)\delta(p,r)R_{is}+R_{jr}+\delta(p,j)R_{ir})  \\
\nonumber
&&+\frac{\sqrt{-1}}{2^{2+\delta(r,s)}}(\delta(p,r)\delta(q,j)R_{si}+R_{rj}+\delta(q,r)R_{sj}).
\end{eqnarray*}
\begin{itemize}
  \item If $i=p=j\leq q$ and $r<s=q$, then
  \[\Gamma_{I,J}^K=\frac{\sqrt{-1}}{8}(R_{jr}+R_{ir})+\frac{\sqrt{-1}}{4}R_{rj}=\frac{\sqrt{-1}}{2}R_{jr}.\]

  \item If $i=p=j$ and $r=s=q$, then
  $\Gamma_{I,J}^K=\frac{\sqrt{-1}}{2}R_{jr}$.

  \item If $i=p<j$ and $r<s=q$, then
  \begin{eqnarray*}
  \Gamma_{I,J}^K&=& \frac{\sqrt{-1}}{4}(\delta(q,j)\delta(p,r)R_{is}+R_{jr})+ \frac{\sqrt{-1}}{4}(\delta(p,r)\delta(q,j)R_{is}+R_{rj})\\
    &=&\frac{\sqrt{-1}}{2}(\delta(q,j)\delta(p,r)R_{pq}+R_{jr})\\
    &=&\left\{\begin{array}{ll}
\sqrt{-1}R_{pq} &\mbox{ if } j=q\text{ and }r=p,\\
\frac{\sqrt{-1}}{2}R_{jr} &\text{ otherwise}.
\end{array}\right.
  \end{eqnarray*}

\item If $i=p<j$ and $r=s=q$, then $\Gamma_{I,J}^K=\frac{\sqrt{-1}}{2}R_{jr}$.
\end{itemize}
At last, we complete the proof of Lemma 2.1.

\end{document}